\documentclass[12pt]{amsart}
\usepackage{amscd}
\usepackage{amsmath}
\usepackage{stmaryrd}
\usepackage{amssymb}
\usepackage{units}
\usepackage[all]{xy}
\def\frk{\frak}               

\def\Phi{{\frk n}}
\def\Phi{{\frk N}}
\def\opn#1#2{\def#1{\operatorname{#2}}} 
\opn\chara{char} \opn\length{\ell} \opn\pd{pd} \opn\rk{rk}
\opn\projdim{proj\,dim} \opn\injdim{inj\,dim} \opn\rank{rank}
\opn\depth{depth} \opn\sdepth{sdepth} \opn\fdepth{fdepth}
\opn\grade{grade} \opn\height{height} \opn\embdim{emb\,dim}
\opn\codim{codim}  \opn\min{min} \opn\max{max}

\opn\Tr{Tr} \opn\bigrank{big\,rank}
\opn\superheight{superheight}\opn\lcm{lcm}
\opn\trdeg{tr\,deg}
\opn\reg{reg} \opn\lreg{lreg} \opn\ini{in} \opn\lpd{lpd}
\opn\size{size}
\opn\div{div} \opn\Div{Div} \opn\cl{cl} \opn\Cl{Cl}
\opn\Spec{Spec} \opn\Supp{Supp} \opn\supp{supp} \opn\Sing{Sing}
\opn\Ass{Ass} \opn\Min{Min}
\opn\Ann{Ann} \opn\Rad{Rad} \opn\Soc{Soc}
\opn\Im{Im} \opn\Ker{Ker} \opn\Coker{Coker} \opn\Am{Am}
\opn\Hom{Hom} \opn\Tor{Tor} \opn\Ext{Ext} \opn\End{End}
\opn\Aut{Aut} \opn\id{id}  \opn\deg{deg}

\opn\nat{nat}
\opn\pff{pf}
\opn\Pf{Pf} \opn\GL{GL} \opn\SL{SL} \opn\mod{mod} \opn\ord{ord}
\opn\Gin{Gin} \opn\Hilb{Hilb}
\opn\aff{aff} \opn\con{conv} \opn\relint{relint} \opn\st{st}
\opn\lk{lk} \opn\cn{cn} \opn\core{core} \opn\vol{vol}
\opn\link{link} \opn\star{star}
\opn\gr{gr}

\def\pot#1#2{#1[\kern-0.28ex[#2]\kern-0.28ex]}

\opn\dirlim{\underrightarrow{\lim}}
\opn\inivlim{\underleftarrow{\lim}}
\let\to=\rightarrow

\def\Implies{\ifmmode\Longrightarrow \else
        \unskip${}\Longrightarrow{}$\ignorespaces\fi}
\def\implies{\ifmmode\Rightarrow \else
        \unskip${}\Rightarrow{}$\ignorespaces\fi}
\def\iff{\ifmmode\Longleftrightarrow \else
        \unskip${}\Longleftrightarrow{}$\ignorespaces\fi}

\let\:=\colon
\newtheorem{Theorem}{Theorem}[]
\newtheorem{Lemma}[Theorem]{Lemma}
\newtheorem{Corollary}[Theorem]{Corollary}
\newtheorem{Proposition}[Theorem]{Proposition}
\theoremstyle{definition}

\newtheorem{Example}[Theorem]{Example}

\let\epsilon\varepsilon
\let\phi=\varphi
\let\kappa=\varkappa
\def\qed{\ifhmode\textqed\fi
      \ifmmode\ifinner\quad\qedsymbol\else\dispqed\fi\fi}
\def\textqed{\unskip\nobreak\penalty50
       \hskip2em\hbox{}\nobreak\hfil\qedsymbol
       \parfillskip=0pt \finalhyphendemerits=0}
\def\dispqed{\rlap{\qquad\qedsymbol}}

\opn\dis{dis}
\def\pnt{{\raise0.5mm\hbox{\large\bf.}}}

\opn\Lex{Lex}

\begin{document}

\title{\bf An Easy proof of the General \\ Neron Desingularization in dimension $\leq$ 1}
\author{  Asma Khalid and Zunaira Kosar}
\thanks{ We gratefully acknowledge the support from the ASSMS GC. University Lahore, for arranging our visit to Bucharest, Romania and we are also greatful to the Simion Stoilow Institute of the Mathematics of the Romanian Academy for inviting us.}

\address{Abdus Salam School of Mathematical Sciences,GC University, Lahore, Pakistan.}
\email{asmakhalid768@gmail.com}
\email{ zunairakosar@gmail.com}

\maketitle

\begin{abstract}  In this paper we give an easy proof of the General Neron Desingularization in the frame of regular morphisms between  Artinian local rings and  Noetherian local rings of  dimension   one.

\end{abstract}

\section*{Introduction}

A ring morphism $u : A \rightarrow A'$ has {\em  geometrically regular fibers} if
for all prime ideals $P \in Spec\ A$ and all finite field extensions $K$ of the fraction field of $A/P$ the ring $K \otimes_{A/P} A'/PA'$ is regular. A
flat morphism $u$ is {\em regular} if its fibers are geometrically regular. A regular morphism of finite type is {\em smooth}  and a localization of a smooth algebra is {\em essentially smooth}.

Let $(A,\mathfrak{m})$ be a Noetherian local  ring of dimension $\leq 1$, $A'$ a Noetherian local ring of dimension one and  $u:A\rightarrow A'$  a regular morphism. Let $B=A[Y]/I$, $Y=(Y_1,\ldots,Y_n)$. Then any morphism $v: B\rightarrow A'$ factors through a smooth $A$-algebra $C$, that is $v$ is a composite $A$-morphism $B\rightarrow C\rightarrow A'.$ This is a particular case of the General Neron Desingularization given in \cite{P0}, \cite{S}, \cite{artin}, with a very difficult proof. An algorithmic proof is given in \cite{adi} for the case when $A,A'$ are  domains, and it  is extended  in \cite{FP} for arbitrary Noetherian local rings  of dimension one.

Let $(A,\mathfrak{m})$ be a local Artinian ring, $(A',\mathfrak{m}')$ a Noetherian complete local ring of dimension one such that $k=A/\mathfrak{m}\cong A'/\mathfrak{m}'$, and $u:A\rightarrow A'$ be a regular morphism. Suppose that  $k\subset A$. Then $\bar {A}'=A'/\mathfrak{m}A'$ is a discrete valuation ring (shortly a DVR). Choose $x\in A'$ such that its class modulo $\mathfrak{m}A'$ is a local parameter of $\bar A'$, that is, it generates  $\mathfrak{m}'\bar A'$.
Let $B=A[Y]/I$, $Y=(Y_1,\ldots,Y_n)$.

The purpose of this paper is to give an easy proof to the following theorem.

\begin{Theorem}\label{main}
Then any morphism $v: B\rightarrow A'$ factors through a smooth $A$-algebra $C$, that is, $v$ is a composite $A$-morphism $B\rightarrow C\rightarrow A'.$
\end{Theorem}

 The proof is illustrated in Example \ref{ex}.

\section{Preliminaries}

A ring morphism $h: A \rightarrow B$ is called \emph{quasi-smooth} \cite{S} if for any $A$-algebra $D$ and an ideal $I \subset D$ with $I^2 = 0$, any $A$-morphism $B \rightarrow D/I$ lifts to an $A$-morphism $B \rightarrow D$. A finitely presented quasi-smooth morphism is  smooth. Let $u : A \rightarrow B$ be a quasi smooth morphism and $C$ an $A$-algebra. Then $C\otimes u$ is also quasi smooth which is known as the \emph{Base change} property.

Let $A$ be a DVR, $x$ a local parameter of $A$, $A' = \hat A$ its completion and $B = A[Y ]/I$, $Y = (Y_1,\ldots,Y_n)$ a finite type $A$-algebra. If $f = (f_1,\ldots, f_r)$,
$r\leq n$ is a system of polynomials from $I$ then we consider an $r \times r$-minor $M$ of the Jacobian matrix $(\partial f_i/\partial Y_j)$. Let
$c \in N$. Suppose that there exist an $A$- morphism $v : B \rightarrow A'/(x^{2c+1})$ and $N \in ((f) : I)$ such that $v(NM) \notin (x)^{c}/(x^{2c+1})$,
where for simplicity we write $v(NM)$ instead $v(NM + I)$.

\begin{Theorem}\cite{P}
There exists a $B$-algebra $C$ which is smooth over $A$ such that every $A$-morphism $v': B \rightarrow A'$ with $v'\equiv v\ modulo\ x^{2c+1}$ (that is $v'(Y ) \equiv v(Y )\ modulo\ x^{2c+1}$) factors through $C$.
\end{Theorem}

\begin{Corollary}\label{corolary2}
Any $A$-morphism $v:B\to A'$ factors through a smooth $A$-algebra $C$.
\end{Corollary}
This corollary follows also from the classical N\'eron desingularization \cite{N}.
\begin{Theorem}\cite{FP}
Let $u : A \rightarrow A'$ be a regular morphism of one dimensional Noetherian local rings and $B$ a finite type $A$-algebra. Suppose that $A$ contains its residue field. Then any $A$-morphism $v : B \rightarrow A'$
factors through a smooth $A$-algebra $C$, that is $v$ is a composite $A$-morphism $B \rightarrow C \rightarrow A'$.
\end{Theorem}

\section{Proof of Theorem \ref{main} when $A=k$}

Then $A'$ is a complete DVR and has the form $A'=k[[x]]$. Let $A_1=k[x]_{(x)}$ and $u_1$ the inclusion $A_1\subset A'$. Then $u_1$ is flat because $A_1$ is principal and $A'$ has no $A$-torsion. The fraction field extension induced by $u_1$ is separable and $A'$ is the completion of $A_1$. It follows that  $u_1$ is regular because $A_1$ is excellent (see e.g. \cite{Mat}).
Let be as above, $B$ a $k$-algebra of finite type, and $v: B\rightarrow A'$ an $A$-morphism. Let $B_1=A_1\otimes_A B$ and $v_1:B_1\rightarrow A'$ be the composite map $a_1\otimes b\mapsto u_1(a_1)\cdot v(b)$.

\begin{Lemma}\label{lemma1}
Suppose that $v_1$ factors through a smooth $A_1$-algebra $\tilde{C}$ then $v$ factors through a smooth $A$-algebra $C$.
\end{Lemma}

\begin{proof}

Since $B_1=A_1\otimes_k B=A_1[Y]/IA_1[Y]$ where $Y=(Y_1, \ldots,Y_n$) and $v_1$ can be factored through a smooth $A_1$-algebra $\tilde{C}$,  we get the following the commutative diagram:
$$
  \begin{xy}\xymatrix{A \ar[dr] \ar[ddr] \ar[r] & A_1 \ar[d]^{\delta} \ar[r]^{u_1} & A'\\
  & B_1 \ar[r]^{\beta} \ar[ur]^{v_1} & \tilde{C} \ar[u]_{w}\\
  & B \ar[u]_{\gamma}}
  \end{xy}
  $$

\noindent where $ \delta: A_1 \to B_1$ is given by $a_1\mapsto a_1\otimes 1$, $a_1\in A_1$, $ \gamma: B \to B_1$ is given by $b \mapsto 1\otimes b$, $b\in B$ and $v_1$ is the composite map $B_1\xrightarrow{\beta} {\tilde C}\xrightarrow{w} A'$. Then
$\tilde{C}$ has the form $(A_1[Z]/(f_1, \ldots , f_r))_{M\tilde{h}}$, $Z=(Z_1,\ldots,Z_s)$ for some $r\times r$-minor $M$ of the Jacobian matrix $\partial f/\partial Z$ and some  $\tilde{h}\in A_1[Z]$   such that $w(\tilde{h})$ is invertible in $A'$. Clearly $\tilde{C}$ is essentially smooth over $k$  but unfortunately not smooth because it is not a $k$-algebra of finite type. Now choose $h\in k[x]\setminus (x)$ such that $f_1, \ldots , f_r$ have coefficients in $k[x,Z]_{h}$. Set $E=(k[x,Z]_{h}/(f_1, \ldots , f_r))_{M,\tilde{h}}$, then $\tilde{C}$ is a localization of $E$ with respect to all $g\in k[x]\setminus (x)$. Thus $\tilde{C}$ is the filtered inductive limit of $E_g$,  $g\in k[x]\setminus (x)$. Let $\varphi_{g}:E_g\rightarrow \tilde{C}$ be the limit maps.

 We claim that $\alpha=\beta\gamma $ factors through a certain $E_g$. Indeed, let $\alpha(\hat{Y_i})=q_i/g_i$ for some $q_i \in E$ and $ g_i \in k[x]\setminus (x)$. Then $\alpha$ factors through $E_g$ for $g=\prod_{i=1}^{n}g_i$. Clearly $v$ is the composite map  $B\rightarrow C=E_g \xrightarrow{\varphi_g} \tilde{C}\xrightarrow{w}A'$, that is, $v$ factors through the smooth $k$-algebra $C$.
\hfill\ \end{proof}

The following proposition is closed to \cite[Corollary 16]{P}.
\begin{Proposition}\label{prop}
Suppose that $A=k$, then $v$ factors through a smooth $k$-algebra $C$.
\end{Proposition}

\begin{proof}
By Corollary \ref{corolary2}, $v_1$ factors through a smooth $A_1=k[x]_{(x)}$-algebra $\tilde C$. Let's say $v_1$ is the composite map $B_1\rightarrow \tilde C\xrightarrow{w} A'$. Then by Lemma \ref{lemma1}, $v$ factors through a smooth $k$-algebra $C$.
\hfill\ \end{proof}

\section{The Proof of Theorem \ref{main}}

Since $A$ is an Artinian local ring  there exists a certain $s$ such that $\mathfrak{m}^s=0$. $A$ has the form $A=k[T]/{\mathfrak a}$, $T=(T_1,\ldots,T_m)$, and the maximal ideal of $A$ is generated by $T$. Then for all $i\in [m]=\{1,\ldots, m\}$, ${T_i}^s\in {\mathfrak a}$ and $A'=k[[x]][T]/(a) \cong A \otimes_k k[[x]]$. Suppose that $B=A[Y]/I$ where $Y=(Y_1,\ldots, Y_n)$, $v:B\rightarrow A'$ and
$v(\hat{Y_i})=\hat{y}_i=\sum_{\alpha \in \mathbb{N}^m, |\alpha|<s}y_{i\alpha}\cdot T^\alpha$,  $T^\alpha={T_1}^{\alpha_1}\cdots {T_m}^{\alpha_m},\ |\alpha|=\alpha_1+\cdots+\alpha_m\ < s$ and $y_{i\alpha}\in k[[x]]$. We have the following commutative diagram.

$$
  \begin{xy}\xymatrix{k \ar[d] \ar[r] & \bar{B} \ar[r] &\bar{B}_1 \ar[r]^{\bar{v}_1} & k[[x]]\\
  A \ar[r] & A\otimes_{k}\bar{B} \ar[r] & A\otimes_{k}\bar{B}_1 \ar[r] & A\otimes_{k}k[[x]]=A' \ar[u]}
  \end{xy}
  $$

\noindent where $\bar{B}=B/ \mathfrak{m} B$. Set $\overline{B_1}=k[(y_{i\alpha})_\alpha]\subset k[[x]]$ and  let $\overline{v_1}$ be this inclusion. Then $v$ factors through $B_1=A\otimes_k {\bar B}_1\subset A'$, that is $v$ is the composite map $B\xrightarrow{q} B_1\xrightarrow{A\otimes_k{\bar v}_1} A'$, where $q$ is defined by $Y_i\to \sum_{\alpha} T^{\alpha}\otimes y_{i\alpha}$. Applying the Proposition \ref{prop}, we see that
${\bar v}_1$ factors through a smooth $k$-algebra ${\bar C}$.  Then $A\otimes_k{\bar v}_1$ factors through $C=A\otimes_k {\bar C}$ and so $v$ factors through $C$.

\begin{Example} \label{ex}
Let $A={\bf Q}[t]/(t^2)$, $A'={\bf Q}[[x]][t]/(t^2)$, $A_1={\bf Q}[x]_{(x)}[t]/(t^2)$,
$B=A[Y_1,Y_2]/(Y_1^3-Y_2^2)$ and $u_1,v_1$ two formal power series from ${\bf Q}[[x]]$
 which are algebraically independent over ${\bf Q}(x)$ and $u_1(0)=v_1(0)=1$. By the Implicit Function
 Theorem there exists $u_2\in k[[x]]$ such that $u_2^2=u_1^3$. Set $v_2=(3/2)xu_1^2v_1u_2^{-1}$,
 ${\hat y}_1=x^2u_1+tv_1$, ${\hat y}_2=x^3u_2+txv_2$. We have
 $f({\hat y}_1,{\hat y_2})=x^6(u_1^3-u_2^2)+tx^4(3u_1^2v_1-2u_2v_2)=0$
and we may define $v:B\to A'$ by $Y\to ({\hat y}_1,{\hat y_2})$. Take $\overline{B_1}=
{\bf Q}[x,x^2u_1,x^3u_2,v_1,xv_2]$. Then ${\bar v}={\bf Q}\otimes_Av: B/tB\to {\bf Q}[[x]]$ factors through
 $\overline{B_1}$. Let $\phi:{\bf Q}[x,Y_{10},Y_{20},Y_{11},Y_{21}]\to \overline{B_1}$ be given by
 $Y_{10}\to x^2u_1$, $Y_{20}\to x^3u_2$, $Y_{11}\to v_1$, $Y_{21}\to xv_2$. Then $\Ker \phi$ contains
 $g_1=Y_{10}^3-Y_{20}^2$, $g_2=3Y_{10}^2Y_{11}-2Y_{20}Y_{21}$. The Jacobian matrix
$\partial g_k/\partial Y_{ij}$ contains a $2\times 2$-minor $M=4Y_{20}^2$ which is mapped by
$\phi$ in $4x^6u_2^2$. Taking $c=7$ we see that $\phi(M)\not \equiv 0$ modulo $x^c$. Note that $\phi$ induces
 the isomorphism $F={\bf Q}[x,Y_{10},Y_{20},Y_{11},Y_{21}]/(g_1,g_2)\cong \overline{B_1}$ because
$\Ker \phi=(g_1,g_2)$ since $\dim F=\dim \overline{B_1}=3$, $u_1,v_1$ being algebraically independent
over ${\bf Q}(x)$.

Take $E={\bf Q}[x,u_1,u_2,v_1,v_2]$. The kernel of the map $\psi:{\bf Q}[x,U_1,U_2,V_1,V_2]\to E$ contains
the polynomials $h_1=U_1^3-U_2^2$, $h_2=3U_1V_1-2U_2V_2$ and as above we see that
${\bf Q}[x,U_1,U_2,V_1,V_2]/(h_1,h_2)\cong E$. Now the Jacobian matrix $\partial h/\partial (U_i,V_i)$
 contains a minor $M'=4U_2^2$ which is mapped by $\psi$ in $4u_2^2\not \in (x)$. Then
$${\bar C}=({\bf Q}[x,U_1,U_2,V_1,V_2]/(h_1,h_2))_{V_2U_2U_1}\cong
 ({\bf Q}[x,U_1,V_1,V_2]/(U_1^3-(9/4)U_1^2V_1^2V_2^{-2}))_{V_2U_1}$$
$$\cong ({\bf Q}[x,U_1,V_1,V_2]/(U_1-(9/4)V_1^2V_2^{-2}))_{V_2}\cong {\bf Q}[x,V_1,V_2]_{V_2}$$
 is smooth over $\bf Q$ and the inclusion $\bar v_1:\overline{B_1}\to {\bf Q}[[x]]$ factors through $\bar C$ because $\bar v_1$ is the composite map
$$\overline{B_1}\cong {\bf Q}[x,Y_{10},Y_{20},Y_{11},Y_{21}]/(g_1,g_2)\xrightarrow{\rho} {\bar C}\xrightarrow{\psi} E\to {\bf Q}[[x]],$$
 where $\rho $ is given by $Y_{10}\to x^2U_1$,$Y_{20}\to x^3U_2$, $V_i\to V_i$. Then $C=A\otimes_{\bf Q}{\bar C}$ is smooth over $A$ and   $v$ factors through $C$ because it is the composite map
$B\xrightarrow{\nu} B_1=A\otimes_{\bf Q}\overline {B_1}\xrightarrow{A\otimes_{\bf Q}{\bar v_1}} A\otimes_{\bf Q}{\bf Q}[[x]]\cong A',$
where $\nu$ is given by $Y_1\to (1\otimes x^2u_1+t\otimes v_1)$, $Y_2\to (1\otimes x^3u_2+t\otimes xv_2)$.
\end{Example}

\end{document}